\documentclass{aims}

\usepackage[latin1]{inputenc}
\usepackage[T1]{fontenc}
\usepackage[american]{babel} 
\usepackage{comment}

 \usepackage[colorlinks=true]{hyperref}
\hypersetup{urlcolor=blue, citecolor=red}
\usepackage[capitalise,nameinlink]{cleveref}
\usepackage{enumitem} 

\def\currentvolume{X}
 \def\currentissue{X}
  \def\currentyear{200X}
   \def\currentmonth{XX}
    \def\ppages{X--XX}
     \def\DOI{10.3934/xx.xx.xx.xx}

\usepackage{amsfonts,amssymb,amsthm,amsopn}
\usepackage[centercolon]{mathtools}
\usepackage{dsfont}

\usepackage[pagewise]{lineno}

\makeatletter
\def\th@plain{%
	\thm@notefont{}
	\itshape 
}
\def\th@definition{%
	\thm@notefont{}
	\normalfont 
}

\newtheorem{theorem}{Theorem}[section]
\newtheorem{proposition}[theorem]{Proposition}

\newtheorem{remark}[theorem]{Remark}
\newtheorem{assumption}[theorem]{Assumption}
\newtheorem{definition}[theorem]{Definition}
\newtheorem{example}[theorem]{Example}

\newcommand{\Norm}[2]{\ensuremath{\left\|#1\right\|_{#2}}}

\newcommand{\R}{\mathbb{R}}
\newcommand{\U}{{U}}
\newcommand{\Y}{{Y}}

\DeclareMathOperator*{\argmin}{arg\,min}

\usepackage[normalem]{ulem}
\newcommand{\rev}[1]{\textcolor{black}{#1}}

\title[On the Nonuniqueness and Instability of Solutions]{On the Nonuniqueness and Instability of Solutions
of Tracking-Type Optimal Control Problems}

\author[C.\ Christof and D.\ Hafemeyer]{}

\subjclass{49J27, 49K40, 49N45, 90C26}
\keywords{optimal control, nonuniqueness, global solutions, nonlinear operators}

 \email{christof@ma.tum.de}
 \email{dominik.hafemeyer@tum.de}
 
\thanks{This research was conducted within the International Research Training Group IGDK 1754,
funded by the German Science Foundation (DFG) and the Austrian Science Fund (FWF) under
project number 188264188/GRK1754}

\thanks{$^*$Corresponding author: Constantin Christof}

\begin{document}
\maketitle

\centerline{\scshape Constantin Christof$^*$}
\medskip
{\footnotesize
\centerline{Department of Mathematics, Technische Universit\"{a}t M\"{u}nchen,}
\centerline{ Boltzmannstr.\ 3, 85748 Garching b.\ M\"{u}nchen, Germany}
}
\medskip

\centerline{\scshape Dominik Hafemeyer}
\medskip
{\footnotesize
\centerline{Department of Mathematics, Technische Universit\"{a}t M\"{u}nchen,}
\centerline{ Boltzmannstr.\ 3, 85748 Garching b.\ M\"{u}nchen, Germany}
}

\bigskip

 \centerline{(Communicated by the associate editor name)}

\begin{abstract}
We study tracking-type optimal control problems that involve a
non-affine, weak-to-weak continuous control-to-state mapping, a desired state $y_d$,
and a desired control $u_d$. It is proved that such problems are always nonuniquely
solvable for certain choices of the tuple $(y_d, u_d)$ and instable in the
sense that the set of solutions (interpreted as a multivalued function of $(y_d, u_d)$)
does not admit a continuous selection.
\end{abstract}

\maketitle


\section{Introduction} 
\label{sec:1}

This paper is concerned with the uniqueness and stability of solutions of tracking-type optimal control problems of the form
\begin{linenomath}\begin{equation}\tag*{\textup{P($y_d$,$u_d$)}}
\label{eq:problem}
\min_{(y,u) \in \Y \times \U}\ \Norm{y-y_d}{\Y}^p +  \Norm{u-u_d}{\U}^p 
\quad
\text{s.t. } y = S(u).
\end{equation}\end{linenomath}
Our standing assumptions on the quantities in \ref{eq:problem} are as follows:

\begin{assumption}~
\label{ass:standing}
\begin{enumerate}[leftmargin=0.865cm,label=(\roman*)]
\item $(\Y, \Norm{\cdot}{\Y} )$ and $(\U, \Norm{\cdot}{\U})$ are uniformly convex, uniformly smooth Banach spaces,
\item $p \in (1, \infty)$ is arbitrary but fixed,
\item $y_d \in \Y$ and $u_d \in \U$ are \rev{problem parameters (the desired state/desired control)},
\item $S\colon\U \to \Y$ is a function that is not affine-linear and satisfies 
\begin{linenomath}\begin{equation*}
u_n \xrightharpoonup{n\rightarrow\infty} u \text{ in }U
 \quad \implies \quad 
S(u_{n}) \xrightharpoonup{n\rightarrow\infty} S(u) \text{ in }Y.
\end{equation*}\end{linenomath}
Here, the symbol ``$\rightharpoonup$'' denotes weak convergence.
\end{enumerate}
\end{assumption}

Due to their simple structure and since they allow to easily construct situations
with known analytic solutions (just choose $u_d := \bar u$ and $y_d := S(\bar u)$ for some given
$\bar u \in U$), tracking-type optimal control problems of the form \ref{eq:problem} are considered very
frequently in the literature -- in particular in the case where the exponent $p$ is equal
to two and the spaces $\Y$ and $\U$ are Hilbert.
Compare, for instance, with \cite{AliDeckelnickHinze2020,Betz2014,Christof2018,Gugat2005,Herty2007,OPTPDE,OPTPDE2}
and the tangible examples in \cref{sec:3} in this context. 
Recently, it was demonstrated in \cite{Pighin2020}
by means of an explicit construction for a boundary control problem with $u_d = 0$
governed by a semilinear elliptic partial differential equation that
problems of the type \ref{eq:problem} can possess multiple global solutions. The aim of this brief
note is to point out that tracking-type optimal control problems which involve a
desired state $y_d$, a desired control $u_d$, and a non-affine, weak-to-weak continuous
control-to-state map $S\colon u \mapsto y$ are indeed always nonuniquely solvable for certain
choices of the tuple $(y_d, u_d)$ -- regardless of whether the control-to-state operator
arises from a partial differential equation, a variational inequality, a differential
inclusion or something else. We further demonstrate that the same effects, that are
responsible for this nonuniqueness of solutions, also cause the problem \ref{eq:problem} to be
instable in the sense that the set of solutions of \ref{eq:problem} (interpreted as a multivalued
map of $(y_d, u_d)$) does not admit a continuous selection. For the main results of this
note, we refer the reader to \cref{th:nonunique,th:nonstable}.

Although, at the end of the day, just consequences of classical results from nonlinear
approximation theory and a simple identification with a metric projection,
we believe that the observations made in this paper are of sufficient interest to justify 
pointing them out and 
making them available in a tangible format -- in particular due to
their very general nature and their potential consequences for, e.g., the study of
turnpike properties, cf.\ the discussion in \cite{Pighin2020} and the references therein. 
We remark that, for the special case of trajectory control problems, 
arguments analogous to those in this note have already been used in \cite{Dontchev1993,Muselli2007,Zolezzi1981}.

\rev{The structure of the remainder of this paper is as follows:
\Cref{sec:2} is concerned with the analysis of the abstract problem \ref{eq:problem}.
Here, we establish the solvability of \ref{eq:problem} for all $(y_d, u_d) \in Y \times U$,
see \cref{prop:existence}, recall the concept of Chebyshev sets, see \cref{def:Chebyshev},
and prove our main results on the nonuniqueness 
and instability of solutions of \ref{eq:problem}, see \cref{th:nonunique,th:nonstable}.
\Cref{sec:2} also contains some comments on our standing assumptions and on the consequences of our analysis,
see \cref{rem:comments}.
In \cref{sec:3}, we demonstrate by means of four examples that \cref{th:nonunique,th:nonstable}
can be applied to a multitude of different problems arising in the field of optimal control, see \cref{ex:1,ex:2,ex:3,ex:4}. 
This section also gives some pointers 
on how \cref{ass:standing} can be verified in practice.
In \cref{sec:4}, we finally conclude the paper with additional remarks on possible extensions of 
our analysis and on cases in which the spaces $Y$ and $U$ lack the properties of uniform convexity and uniform smoothness.}

\section{Nonuniqueness and Instability of Solutions}
\label{sec:2}

\rev{We begin our analysis of the problem \ref{eq:problem} by proving its solvability:}

\begin{proposition}[Existence of Global Minimizers]
\label{prop:existence}
In the situation of \cref{ass:standing}, 
the minimization problem \ref{eq:problem} admits at least one global solution
$(\bar y, \bar u) \in Y \times U$ for every choice of the tuple $(y_d, u_d) \in Y \times U$.
\end{proposition}
\begin{proof}
The claim follows straightforwardly from the direct method of calculus of
variations. Indeed, if we consider a minimizing sequence $\{(y_n, u_n)\}_{n \in \mathbb{N}} \subset Y \times U$
of \ref{eq:problem}, then the sequences $\{y_n\}_{n \in \mathbb{N}} \subset Y$ and $\{u_n\}_{n \in \mathbb{N}} \subset U$ are trivially bounded by the
structure of the objective function of \ref{eq:problem}, and it follows from our assumption of
uniform convexity and the theorems of Milman-Pettis and Banach-Alaoglu, see \cite{Pettis1939}
and \cite[Section V-2]{Yosida1980}, that the spaces $\Y$ and $\U$ are reflexive and that we may extract a subsequence
of $\{(y_n, u_n)\}_{n \in \mathbb{N}}$ (for simplicity denoted by the same symbol) such that 
$\{y_n\}_{n \in \mathbb{N}}$
converges weakly in $\Y$ to some $\bar y \in Y$ and $\{u_n\}_{n \in \mathbb{N}}$ 
converges weakly in $\U$ to some $\bar u \in U$.
Note that the weak-to-weak continuity of $S$ implies that $\bar y = S(\bar u)$ has to hold. 
In combination with the weak lower semicontinuity of continuous and
convex functions, see \cite[Corollary 4.1.14]{BorweinVanderwerff2010}, it now follows immediately that
\begin{linenomath}\begin{equation*}
\begin{aligned}
\inf_{(y,u) \in \Y \times \U,\, y = S(u)}
\Norm{y-y_d}{\Y}^p +  \Norm{u-u_d}{\U}^p 
&=
\lim_{n \to \infty} \Norm{y_n-y_d}{\Y}^p +  \Norm{u_n-u_d}{\U}^p 
\\
&\geq 
\Norm{\bar y-y_d}{\Y}^p +  \Norm{\bar u-u_d}{\U}^p.
\end{aligned}
\end{equation*}\end{linenomath}
This shows that $(\bar y, \bar u)$ is a global solution of \ref{eq:problem} and completes the proof.
\end{proof}

\rev{Having established that \ref{eq:problem} possesses a solution
for all $(y_d, u_d) \in Y \times U$, we can turn our attention to questions of uniqueness. 
For the discussion of this topic, we require the following classical concept, see, e.g., \cite[Section 0]{Klee1961}:}

\begin{definition}[Chebyshev Set]\label{def:Chebyshev}
\rev{
A subset $M$ of a metric space $(Z, \rho)$ is called a Chebyshev set if, for each $z \in Z$,
there exists a unique nearest element $\bar m \in M$, i.e., a unique $\bar m \in M$ satisfying
$\rho(z, \bar m) = \inf\{\rho(z, m) \mid m \in M\}$. }
\end{definition}

\rev{We are now in the position to prove our first main result:}

\begin{theorem}[Nonuniqueness of Global Minimizers]
\label{th:nonunique}
In the situation of \cref{ass:standing}, 
there always exists a tuple $(y_d, u_d) \in Y \times U$ such that the problem
\ref{eq:problem} possesses more than one global solution.
\end{theorem}
\begin{proof}
The main idea of the proof is to identify \ref{eq:problem} with a metric projection problem
onto the graph of the control-to-state mapping $S$, i.e., the set
\begin{linenomath}\begin{equation}
\label{eq:Mdef}
M:=
\left \{
(S(u), u) \mid u \in U
\right \} \subset Y \times U
\end{equation}\end{linenomath}
and to subsequently invoke classical results on the convexity of Chebyshev sets. To
pursue this approach, we argue by contradiction.

Assume that, for each tuple $(y_d, u_d) \in Y \times U$, the problem \ref{eq:problem} possesses precisely one global solution
$(\bar y, \bar u) \in Y \times U$.
Then, the monotonicity of the function $[0, \infty) \ni x \mapsto x^{1/p}  \in [0, \infty)$
implies that, for every $(y_d, u_d)$, the unique global minimizer of \ref{eq:problem} is also the
sole solution of the problem
\begin{linenomath}\begin{equation}
\label{eq:projprob}
\min_{(y, u) \in M}
\Norm{(y, u) - (y_d, u_d)}{\Y \times \U},
\end{equation}\end{linenomath}
where $M$ is the set in \eqref{eq:Mdef} and where $\Norm{\cdot}{\Y \times \U}$ is the norm on 
$Y \times U$ defined by
\begin{linenomath}\begin{equation}
\label{eq:YxUnormdef}
\Norm{(y, u)}{\Y \times \U}
:=
\left ( 
\Norm{y }{\Y}^p +  \Norm{u }{\U}^p 
\right )^{1/p}
\qquad \forall (y,u) \in Y \times U.
\end{equation}\end{linenomath}
Note that the space $Y \times U$ endowed with the norm $\Norm{\cdot}{\Y \times \U}$ is trivially Banach, and that
\cite[Theorem 1]{Clarkson1936} and our assumptions on $Y$ and $U$ imply that $(Y \times U, \Norm{\cdot}{\Y \times \U})$
is uniformly convex. Further, the space $(Y \times U, \Norm{\cdot}{\Y \times \U})$ is also uniformly smooth.
Indeed, from \cite[Theorem 5.5.12]{Megginson1998}, we obtain that the uniform smoothness of the
spaces $(\Y, \Norm{\cdot}{\Y} )$ and $(\U, \Norm{\cdot}{\U})$ is equivalent to the uniform convexity of the duals
$(\Y^*, \Norm{\cdot}{\Y^*} )$ and $(\U^*, \Norm{\cdot}{\U^*})$, and,  
using a standard calculation, it is easy to check that the dual
of $(Y \times U, \Norm{\cdot}{\Y \times \U})$ is isometrically isomorphic to the product space
$Y^* \times U^*$ 
endowed with the norm
\begin{linenomath}\begin{equation*}
\Norm{(y^*, u^*)}{\Y^* \times \U^*}
:=
\left ( 
\Norm{y^* }{\Y^*}^{p/(p-1)} +  \Norm{u^* }{\U^*}^{p/(p-1)}
\right )^{(p-1)/p}
\qquad \forall (y^*,u^*) \in Y^* \times U^*.
\end{equation*}\end{linenomath}
In combination with \cite[Theorem 1]{Clarkson1936}, the above implies that 
$(Y^* \times U^*, \Norm{\cdot}{\Y^* \times \U^*})$ 
is uniformly convex and, by  \cite[Proposition 5.2.7 and Theorem 5.5.12]{Megginson1998},
that the space $(Y \times U, \Norm{\cdot}{\Y \times \U})$ is uniformly smooth as claimed.

Taking into account all of the above and the structure of the problem \eqref{eq:projprob}, we
may conclude that, in the considered situation and under the assumption that the
problem \ref{eq:problem} is uniquely solvable for all $(y_d, u_d) \in Y \times U$, the metric projection in
the uniformly convex and uniformly smooth Banach space  $(Y \times U, \Norm{\cdot}{\Y \times \U})$ onto the
set $M$ defined in \eqref{eq:Mdef} is well-defined and single-valued everywhere. In other words,
$M$ is a Chebyshev set in $(Y \times U, \Norm{\cdot}{\Y \times \U})$ \rev{in the sense of \cref{def:Chebyshev}}. From the
weak-to-weak continuity of the control-to-state mapping $S$, we further obtain that
every sequence $\{(y_n, u_n)\}_{n \in \mathbb{N}} \subset M$ 
that converges weakly in $Y \times U$ to some $(\tilde y, \tilde u)$
has to satisfy
\begin{linenomath}\begin{equation*}
\tilde y \xleftharpoonup{n\rightarrow\infty} y_{n} = S(u_{n}) \xrightharpoonup{n\rightarrow\infty} S(\tilde u).
\end{equation*}\end{linenomath}
The set $M$ is thus not only Chebyshev but also weakly closed and we may invoke
\cite[Corollary 4.2]{Klee1961} to deduce that $M$ has to be convex, i.e., we have
\begin{linenomath}\begin{equation}
\label{eq:convexityrel}
\lambda (y_1, u_1) + (1-\lambda) (y_2, u_2) = ( \lambda S(u_1) + (1-\lambda) S(u_2), \lambda u_1 + (1-\lambda) u_2 ) \in M
\end{equation}\end{linenomath}
for all $\lambda \in [0, 1]$ and all $(y_1, u_1), (y_2, u_2) \in M$. Due to the definition 
of $M$, \eqref{eq:convexityrel} can only be true if 
\begin{linenomath}\begin{equation}
\label{eq:randomeq2635}
 S( \lambda u_1 + (1-\lambda) u_2 ) = \lambda S(u_1) + (1-\lambda) S(u_2)
\end{equation}\end{linenomath}
holds for all $\lambda \in [0, 1]$ and all $u_1, u_2 \in U$.
This property, however, implies in combination with our assumptions on $S$
that the map $L(\cdot) := S(\cdot) - S(0)$
is linear and continuous as a function from $U$ to $Y$. Indeed, 
for every arbitrary but fixed $u \in U$, \eqref{eq:randomeq2635} yields
\begin{linenomath}\begin{equation*}
L(\alpha u) = S(\alpha u + (1 - \alpha)0) - S(0) = \alpha S(u) - \alpha S(0) = \alpha L(u)\quad \forall \alpha \in [0, 1]
\end{equation*}\end{linenomath}
and 
\begin{linenomath}\begin{equation*}
\alpha L(u) = \alpha L\left ( \frac1\alpha \alpha u \right ) = L(\alpha u)\quad \forall \alpha \in (1, \infty).
\end{equation*}\end{linenomath}
From these equations, it readily follows that 
\begin{linenomath}\begin{equation*}
\begin{aligned}
L(u_1+u_2) &= S\left (\frac{1}{2} (2u_1) + \frac{1}{2} (2u_2) \right ) - S(0) 
= \frac{1}{2} S(2u_1) + \frac{1}{2} S(2u_2) - S(0) 
\\
&= \frac{1}{2} L(2u_1) + \frac{1}{2} L(2u_2) 
= L(u_1) + L(u_2)\qquad \forall u_1, u_2 \in U. 
\end{aligned}
\end{equation*}\end{linenomath}
In particular, $L(-u) = -L(u)$ for all $u \in U$, and we may conclude that 
\begin{linenomath}\begin{equation*}
L(\alpha u_1 + u_2) = L(\alpha u_1) + L(u_2) = \alpha L(u_1) + L(u_2) \qquad \forall u_1, u_2 \in U\quad \forall \alpha \in \R.
\end{equation*}\end{linenomath}
The function $L\colon U \to Y$ is thus linear as claimed and, since the weak closedness of the set $M$
immediately yields the closedness of the graph of $L$ in $Y \times U$, also continuous
by the closed graph theorem, see, e.g., \cite[Section II-6]{Yosida1980}.

In summary, we now arrive at the conclusion that the map $S$ has to be an affine-linear function. 
This contradicts our standing assumptions and establishes that 
\ref{eq:problem} cannot possess precisely one solution for all $(y_d, u_d) \in Y \times U$.
As we already know that \ref{eq:problem} possesses at least one solution for each $(y_d, u_d)$ by \Cref{prop:existence},
the assertion of the theorem now follows immediately.
\end{proof}

Next, we address the issue of instability: 

\begin{theorem}[Nonexistence of a Continuous Selection of Minimizers]
\label{th:nonstable}
In the situation of \cref{ass:standing}, 
there always exist a tuple $(y_d, u_d) \in Y \times U$, 
sequences  $\{(y_{d,n}, u_{d,n})\}_{n \in \mathbb{N}} \subset Y \times U$
and $\{(y_{d,n}', u_{d,n}')\}_{n \in \mathbb{N}} \subset Y \times U$,
and elements $(\bar y, \bar u) \in Y \times U$ and  $(\bar y', \bar u') \in Y \times U$
such that the following is true:
\begin{enumerate}[label=(\roman*)]
\item $\{(y_{d,n}, u_{d,n})\}_{n \in \mathbb{N}}$ and $\{(y_{d,n}', u_{d,n}')\}_{n \in \mathbb{N}}$ 
converge strongly in $Y \times U$ to $(y_d, u_d)$,
\item \rev{$(\bar y,\bar u)$ is the unique solution of the problem 
\hyperref[eq:problem]{\textup{P($y_{d,n}$,$u_{d,n}$)}} for all $n \in \mathbb{N}$, i.e.,} 
\begin{linenomath}\begin{equation*}
\{(\bar y, \bar u)\} =  \argmin_{(y,u) \in Y \times U, \, y = S(u)} \Norm{y-y_{d,n}}{\Y}^p +  \Norm{u-u_{d,n}}{\U}^p \qquad \forall n \in \mathbb{N},
\end{equation*}\end{linenomath}
\item  \rev{$(\bar y',\bar u')$ is the unique solution of the problem 
\hyperref[eq:problem]{\textup{P($y_{d,n}'$,$u_{d,n}'$)}} for all $n \in \mathbb{N}$, i.e.,} 
\begin{linenomath}\begin{equation*}
\{(\bar y', \bar u')\} =  \argmin_{(y,u) \in Y \times U, \, y = S(u)} \Norm{y-y_{d,n}'}{\Y}^p +  \Norm{u-u_{d,n}'}{\U}^p \qquad \forall n \in \mathbb{N},
\end{equation*}\end{linenomath}
\item $(\bar y, \bar u) \neq (\bar y', \bar u')$.
\end{enumerate}
\end{theorem}

\begin{proof}
In the considered situation, we obtain 
from exactly the same arguments as in the proof of \cref{th:nonunique}
that \ref{eq:problem} is equivalent to the projection problem \eqref{eq:projprob}
and from \cref{th:nonunique} itself that there exists a tuple  $(y_d, u_d) \in Y \times U$
such that \ref{eq:problem} (and thus also \eqref{eq:projprob}) 
possesses two nonidentical global solutions $(\bar y, \bar u) \in Y \times U$ and $(\bar y', \bar u') \in Y \times U$.
Define 
\begin{linenomath}\begin{equation*}
(y_{d,t}, u_{d,t}) := t(\bar y, \bar u) +  (1 - t)(y_d, u_d)\qquad \forall t \in (0, 1)
\end{equation*}\end{linenomath}
and
\begin{linenomath}\begin{equation*}
(y_{d,t}', u_{d,t}') := t(\bar y', \bar u') + (1 - t)(y_d, u_d)\qquad \forall t \in (0, 1).
\end{equation*}\end{linenomath}
Then, the uniform convexity of the space 
$(Y \times U, \Norm{\cdot}{\Y \times \U})$ (with $\Norm{\cdot}{\Y \times \U}$ defined as in \eqref{eq:YxUnormdef}, 
see again \cite[Theorem 1]{Clarkson1936}) and exactly the same calculations 
as in the proof of \cite[Theorem 2.1]{Kainen2000} yield that 
\begin{linenomath}\begin{equation*}
\{(\bar y, \bar u)\} = 
\argmin_{(y, u) \in M}
\Norm{(y, u) - (y_{d,t}, u_{d,t})}{\Y \times \U}\qquad \forall t \in (0, 1)
\end{equation*}\end{linenomath}
and 
\begin{linenomath}\begin{equation*}
\{(\bar y', \bar u')\} = 
\argmin_{(y, u) \in M}
\Norm{(y, u) - (y_{d,t}', u_{d,t}')}{\Y \times \U}\qquad \forall t \in (0, 1)
\end{equation*}\end{linenomath}
holds, where $M$ is the set in \eqref{eq:Mdef}.
To establish the assertion of the theorem, it now suffices to
choose an arbitrary sequence $\{t_n\}_{n \in \mathbb{N}} \subset (0, 1)$ with $t_n \to 0$,
to define  $(y_{d,n}, u_{d,n}) :=  (y_{d,t_n}, u_{d,t_n})$ and $(y_{d,n}', u_{d,n}') :=  (y_{d,t_n}', u_{d,t_n}')$
for all $n \in \mathbb{N}$,
and to again exploit the equivalence between the problems \ref{eq:problem} and \eqref{eq:projprob}.
\end{proof}

Some remarks regarding the last two results are in order: 

\begin{remark}\label{rem:comments}~
\begin{enumerate}[label=(\roman*)]

\item 
The assumption that both the desired state $y_d$ and the desired control $u_d$ can be chosen 
at will in \cref{th:nonunique} cannot be dropped. If, e.g., $u_d$ is fixed to be zero, 
then it is perfectly possible that a problem of the type \ref{eq:problem} is
uniquely solvable for all $y_d \in Y$ even if the control-to-state mapping $S$ is non-affine. 
An example of such a configuration can be found in \cite[Corollary~5.3]{ChristofSSC2020}. 

\item 
The nonuniqueness of global minimizers in \cref{th:nonunique} implies that 
numerical solution algorithms for problems of the type \ref{eq:problem}
may produce sequences of iterates with several accumulation points
and that termination criteria which 
consider the distance between successive iterates cannot be expected to reliably detect stationarity.
The instability of the solutions in \cref{th:nonstable} further shows that 
numerical errors and small inaccuracies in the problem data may prevent a proper identification of a global optimum.

\item 
\Cref{th:nonstable} shows that, in the situation of \cref{ass:standing}, every function $F\colon Y \times U \to U$
with the property 
\begin{linenomath}\begin{equation*}
\qquad F(y_d, u_d) \in \argmin_{u \in U} \ \Norm{S(u)-y_d}{\Y}^p +  \Norm{u-u_d}{\U}^p  \quad \forall (y_d, u_d) \in Y \times U
\end{equation*}\end{linenomath}
is discontinuous. There thus does not exist a continuous selection from the set of optimal 
controls of \ref{eq:problem} (in the sense of set-valued analysis, cf.\ \cite{Brown1989}). \Cref{th:nonstable} further illustrates that, 
in the presence of nonlinearity, adding a Tikhonov-type regularization term to an objective function
may fail to properly regularize an inverse problem. 

\item \rev{If, instead of \ref{eq:problem}, we consider a problem of the form 
\begin{linenomath}\begin{equation}\tag*{\textup{P($y_d$,$u_d$,$\nu$)}}
\label{eq:problem2}
\min_{(y,u) \in \Y \times \U}\ \Norm{y-y_d}{\Y}^p +  \nu \Norm{u-u_d}{\U}^p 
\quad
\text{\emph{s.t.} } y = S(u)
\end{equation}\end{linenomath}
with some $(\Y, \Norm{\cdot}{\Y} )$, $(\U, \Norm{\cdot}{\U})$, $p \in (1, \infty)$, $y_d \in \Y$, $u_d \in \U$,
and $S\colon\U \to \Y$ as in \cref{ass:standing} and a regularization parameter $\nu > 0$, then, by redefining the norm on $U$, 
this problem can be recast in the form \ref{eq:problem} and the results in 
\cref{th:nonunique,th:nonstable} carry over immediately, cf.\ \cref{ex:1,ex:2,ex:3,ex:4}.
We remark that, if a problem of the type \ref{eq:problem2} is given
that possesses more than one global solution for a certain triple $(y_d, u_d, \nu)$, 
then it is not always possible to remove this nonuniqueness 
by driving the regularization parameter $\nu$ to infinity. Such effects occur, for instance, in the presence of symmetries 
as one may easily check by means of the prototypical example
\begin{linenomath}\begin{equation*}
\min_{y \in \R, \, u \in \R} (y - 1)^2 + \nu u^2 \\
\quad
\text{\emph{s.t.} } y = |u|, 
\end{equation*}\end{linenomath}
which is clearly of the form \ref{eq:problem2} with $Y :=U:=\R$, $\Norm{\cdot}{\Y}  := \Norm{\cdot}{\U}  := |\cdot |$, $p:=2$,
$y_d := 1$, $u_d := 0$, and $S(u) := |u|$ and which possesses the two optimal controls $\bar u_1 := -(1 + \nu)^{-1}$
and $\bar u_2 := (1 + \nu)^{-1}$ for all $\nu > 0$. 
}
\end{enumerate}
\end{remark}

\section{\rev{Tangible Examples}}
\label{sec:3}
\rev{To illustrate that \cref{th:nonunique,th:nonstable} can be applied to a broad range of 
tracking-type optimal control problems, we next discuss some tangible examples.}
(Note that the following list is far from exhaustive.)

\begin{example}[Finite-Dimensional Tracking-Type Problems]
\label{ex:1}
Consider a finite-dimensional optimization problem of the form
\begin{linenomath}\begin{equation}
\label{eq:exprob1}
\min_{y \in \R^l, \, u \in \R^m} \frac{1}{2} (y - y_d)^T A (y-y_d) + \frac{\nu}{2} (u - u_d)^T B (u - u_d) \\
\quad
\text{\emph{s.t.} } y = S(u)
\end{equation}\end{linenomath}
with some $l, m \in \mathbb{N}$, an arbitrary but fixed Tikhonov parameter $\nu > 0$,
symmetric positive definite matrices $A \in \R^{l \times l}$ and $B \in \R^{m \times m}$,
vectors $y_d \in \R^l$ and $u_d \in \R^m$,
and a non-affine, continuous mapping $S\colon \R^m \to \R^l$. Then, by defining
\begin{linenomath}\begin{equation*}
Y := \R^l,\quad 
 \Norm{y}{Y} := \left ( \frac{1}{2}y^T A y\right )^{1/2},
\quad
U:= \R^m,
\quad
 \Norm{u}{U} := \left ( \frac{\nu}{2}u^T B u\right )^{1/2},
\quad p := 2,
\end{equation*}\end{linenomath}
we can recast \eqref{eq:exprob1} as a problem of the form \ref{eq:problem} that 
satisfies all of the conditions in \cref{ass:standing} (as one may easily check). \Cref{th:nonunique,th:nonstable} are thus applicable
to \eqref{eq:exprob1}, 
and we may deduce that there exist choices of the tuple $(y_d, u_d)$ for which \eqref{eq:exprob1} 
possesses more than one global solution and that the solution set of \eqref{eq:exprob1} does not admit a continuous selection.
Note that problems of the type \eqref{eq:exprob1} arise very frequently in optimal control 
when a continuous tracking-type problem is discretized, e.g., by means of finite elements,
cf.\ \cite[Section 5.1]{Christof2018} and \cite[Sections 4.3, 5.3]{Hafemeyer2020PhD}.
\end{example}

\begin{example}[Optimal Control of a Nonsmooth Semilinear Elliptic PDE]
\label{ex:2}
Consider an optimal control problem of the form 
\begin{linenomath}\begin{equation}
\label{eq:ex2}
\begin{aligned}
&\mathrm{min} &&\frac{1}{2} \Norm{y - y_d }{L^2(\Omega)}^2 
+ \frac{\nu}{2} \Norm{u - u_d }{L^2(\Omega)}^2 \\
&\mathrm{w.r.t.} &&y \in H_0^1(\Omega),\quad u \in L^2(\Omega),
\\
&\mathrm{\,s.t. } && - \Delta y + \max(0, y) = u \text{ in } \Omega,
\end{aligned}
\end{equation}\end{linenomath}
where $\Omega \subset \R^m$, $m \in \mathbb{N}$, is a bounded domain, 
$y_d \in L^2(\Omega)$ and $u_d \in L^2(\Omega)$ are given,
${\nu > 0}$ is an arbitrary but fixed Tikhonov parameter, $L^2(\Omega)$ and $H_0^1(\Omega)$ are defined as in \cite{Attouch2006}, 
$\Delta$ is the distributional Laplacian, and the function 
$\max(0, \cdot) \colon \R \to \R$ acts as a Nemytskii operator. 
Then, it follows from \cite[Proposition 2.1, Corollary~3.8]{Christof2018}
that \eqref{eq:ex2} possesses a well-defined and 
weak-to-weak continuous control-to-state mapping
$S\colon L^2(\Omega) \to L^2(\Omega)$, $u \mapsto y$.
Further, the map $S$ is also non-affine. Indeed, 
if we choose an arbitrary but fixed $z \in H_0^1(\Omega) \cap H^2(\Omega)$ that is 
positive almost everywhere in $\Omega$ (such a $z$ exists by \cite[Lemma A.1]{Christof2018})
and if we define
\begin{linenomath}\begin{equation*}
u_1 := 2(-\Delta z + z) \in L^2(\Omega)\qquad
\text{and}
\qquad
u_2 := 2\Delta z \in L^2(\Omega),
\end{equation*}\end{linenomath}
then we clearly have $S(u_1) = 2z$, $S(u_2) = -2z$, and 
\begin{linenomath}\begin{equation*}
\frac12 S(u_1) + \frac12 S(u_2) = 0 \neq S(z) = S\left (\frac12 u_1 + \frac12 u_2\right ),
\end{equation*}\end{linenomath}
where the inequality $0 \neq S(z)$ follows immediately from the PDE in \eqref{eq:ex2}
and our assumption $z > 0$ a.e.\ in $\Omega$. Since
\eqref{eq:ex2} can be recast as a problem of
the form \ref{eq:problem} (with $Y := L^2(\Omega)$, $U := L^2(\Omega)$, $p:= 2$,
and appropriately rescaled norms)
and since Hilbert spaces 
are trivially uniformly convex and uniformly smooth,
we may now conclude that 
the optimal control problem \eqref{eq:ex2} satisfies all of the conditions in \cref{ass:standing}. 
\Cref{th:nonunique,th:nonstable}
are thus applicable and it follows that 
\eqref{eq:ex2} is not uniquely solvable for certain choices of the tuple 
$(y_d, u_d) \in L^2(\Omega) \times L^2(\Omega)$ and
that the solution set of \eqref{eq:ex2} does not admit a continuous selection. Note that the above setting is precisely 
the one considered in \cite{Christof2018}.
\end{example}

\begin{example}[$\boldsymbol{L^p}$-Boundary Control for a Signorini-Type VI]
\label{ex:3}
Consider an optimal control problem of the form
\begin{linenomath}\begin{equation}
\label{eq:ex3}
\begin{aligned}
&\mathrm{min} &&\frac{1}{p} \Norm{y - y_d }{L^p(\Omega)}^p 
+ \frac{\nu}{p} \Norm{u - u_d }{L^p(\partial \Omega)}^p \\
&\mathrm{w.r.t.} &&y \in H^1(\Omega),\quad u \in L^p(\partial \Omega),
\\
&\mathrm{\,s.t. } && y \in K,\quad 
\int_\Omega \nabla y \cdot \nabla (v - y) + y (v - y)\mathrm{d}x \geq \int_{\partial \Omega} u (v - y) \mathrm{d}s\quad \forall v \in K,
\end{aligned}
\end{equation}\end{linenomath}
where $\Omega \subset \R^m$, $m \in \mathbb{N}$, is a bounded Lipschitz domain with boundary $\partial\Omega$, 
${\nu > 0}$ is an arbitrary but fixed Tikhonov parameter,
${y_d \in L^p(\Omega)}$ and $u_d \in L^p(\partial \Omega)$ are given,
$p$ is an exponent that satisfies $p \in [2, \infty)$ for $m \leq 2$ and 
$p \in [2, 2m/(m-2)]$ for $m \geq 3$,
$L^p(\partial \Omega)$, $L^p(\Omega)$, and $H^1(\Omega)$ are defined as in \cite{Attouch2006}, 
$\nabla$ is the weak gradient, and $K$ is the set
of all elements of $H^1(\Omega)$ whose trace is nonnegative a.e.\ on $\partial \Omega$.
Then, using \cite[Theorem II-2.1]{KinderlehrerStampacchia1980}, the Sobolev embeddings, see 
\cite[Theorem~2-3.4]{Necas2012}, and the 
compactness of the trace operator, see \cite[Theorem~2-6.2]{Necas2012},
it is easy to check that the elliptic variational inequality in \eqref{eq:ex3} possesses a well-defined 
and weak-to-weak continuous solution operator
$S\colon L^p(\partial\Omega)\to H^1(\Omega) \hookrightarrow L^p(\Omega)$, $u \mapsto y$.
To see that this $S$ is non-affine, we note that, for every a.e.-positive control $u \in L^p(\partial \Omega)$,
the trace of $S(u)$ has to be positive a.e.\ on a set of positive surface measure. 
Indeed, if the latter was not the case for an a.e.-positive control $u$,
then the variational inequality in \eqref{eq:ex3} and the inclusion $H_0^1(\Omega) \subset K$ 
would imply that $y = S(u) \in H^1(\Omega)$
is also the solution of 
\begin{linenomath}\begin{equation*}
-\Delta y + y = 0 \text{ in } \Omega,\qquad y = 0 \text{ on }  \partial \Omega.
\end{equation*}\end{linenomath}
This, however, would yield $y = 0$ and, as a consequence, 
\begin{linenomath}\begin{equation*}
0  \geq \int_{\partial \Omega} u v\,\mathrm{d}s = \int_{\partial \Omega}| u v|\,\mathrm{d}s\quad \forall v \in K
\end{equation*}\end{linenomath}
which is a contradiction. The trace of $S(u)$ thus has to be positive on a non-negligible 
subset of $\partial \Omega$ for all a.e.-positive $u \in L^p(\partial \Omega)$ as claimed. 
Since we trivially have $S(0) = 0$ and since $S(u)$ has to be an element of $K$ for all $u$ by the definition of $S$,
it now follows immediately that $S(u) + S(-u) \neq S(0)$ holds for all $u \in L^p(\partial \Omega)$
that are positive a.e.\ on $\partial \Omega$. In combination with our previous 
observations on $S$ and the fact that $L^q$-spaces are uniformly convex and 
uniformly smooth for $1 < q < \infty$ (see \cite[Theorems 5.2.11, 5.5.12]{Megginson1998}), 
this shows that \eqref{eq:ex3} satisfies 
the conditions in \cref{ass:standing} (with $Y := L^p(\Omega)$, $U := L^p(\partial \Omega)$,
and again appropriately rescaled norms). We may thus again invoke 
\Cref{th:nonunique,th:nonstable} to deduce that \eqref{eq:ex3} 
is not uniquely solvable for certain tuples $(y_d, u_d) \in L^p(\Omega) \times L^p(\partial \Omega)$
and that the solution set of \eqref{eq:ex3} does not admit a continuous selection. 
\end{example}

\begin{example}[Distributed Control of the Parabolic Obstacle Problem]
\label{ex:4}
Consider an optimal control problem of the form
\begin{linenomath}\begin{equation}
\label{eq:ex4optprob}
\begin{aligned}
		&\mathrm{min} &&\frac{1}{2} \Norm{y(T) - y_d }{L^2(\Omega)}^2 + \frac{\nu}{2} \Norm{u - u_d }{L^2(0, T; L^2(D))}^2 \\
		&\mathrm{w.r.t.} &&y \in L^2(0, T;H_0^1(\Omega)) \cap H^1(0, T;L^2(\Omega)),\quad u \in L^2(0, T; L^2(D)),
\end{aligned}
\end{equation}\end{linenomath}
\rev{that is governed by an evolution variational inequality of the type }
\rev{\begin{linenomath}\begin{equation}
\label{eq:ex4obsprob-def}
\begin{aligned}
	&y \in L^2(0, T;H_0^1(\Omega)) \cap H^1(0, T;L^2(\Omega)),\\
	&y(0) = 0 \text{ a.e.\ in }\Omega,\qquad  y(t) \geq \psi \text{ a.e.\ in }\Omega \text{ for a.a.\ } t \in (0, T),\\
	& 
	\int_0^T \left  \langle  \partial_t y - \Delta y - Bu, v - y \right\rangle  \mathrm{d}t \geq 0 \\
	 &\forall v \in L^2(0, T; H_0^1(\Omega)), \ \ v(t) \geq \psi \text{ a.e.\ in }\Omega 
	\text{ for a.a.\ } t \in (0, T).
\end{aligned}
\end{equation}\end{linenomath}}%
Here, $\Omega \subset \R^m$, $m \in \mathbb{N}$, is supposed to be a bounded domain, 
$D$ is a nonempty, open subset of $\Omega$, $T>0$ is a given final time, $\nu > 0$
is an arbitrary but fixed Tikhonov \mbox{parameter,}
the appearing Lebesgue-, Sobolev-, and Bochner spaces are defined as in \cite{Attouch2006} and \cite{Heinonen2015},
$y_d \in L^2(\Omega)$ and $u_d \in L^2(0, T;L^2(D))$ are given, 
$\psi \in L^2(\Omega)$ is a given function  that satisfies $\psi \leq 0$ a.e.\ in $\Omega$,
$\partial_t$ is the time derivative in the Sobolev-Bochner sense, 
$\Delta$ is the distributional Laplacian, $B$ denotes the canonical embedding of 
$L^2(0, T;L^2(D))$ into $L^2(0, T;L^2(\Omega))$ obtained from an extension by zero, 
and $\left \langle \cdot, \cdot \right \rangle$
denotes the dual pairing in $H_0^1(\Omega)$. 
Then, using \cite[Theorem~1.13, Equation (1.70)]{Barbu1984}, \cite[Theorem~2.3]{Christof2019},
and the lemma of Aubin-Lions, see \cite[Theorem 10.12]{Schweizer2013},
it is easy to check that the variational inequality in \rev{\eqref{eq:ex4obsprob-def}} possesses 
a well-defined, weak-to-weak continuous solution map
$G \colon L^2(0, T;L^2(D)) \to H^1(0, T;L^2(\Omega))$, $u \mapsto y$.
(Note that, in order to apply \cite[Theorem~1.13]{Barbu1984}, one has to define 
the function $\Phi$ appearing in this theorem as in \cite[Equation~(4.9)]{Barbu1984}.)
As $H^1(0, T;L^2(\Omega))$ embeds continuously into $C([0, T];L^2(\Omega))$
by \cite[Theorem 10.9]{Schweizer2013}, 
the above implies that 
 \rev{\eqref{eq:ex4optprob}}
possesses a well-defined, weak-to-weak continuous control-to-state 
(or, in this context, more precisely control-to-observation)
operator 
${S\colon L^2(0,T;L^2(D)) \to L^2(\Omega)}$, $u \mapsto G(u)(T)$, where 
$G(u)(T)$ denotes the value of 
the $C([0, T];L^2(\Omega))$-representative of $G(u)$ at the final time $T$. 
To see that the map $S$ is non-affine, we proceed similarly to \cref{ex:2,ex:3}.
Suppose that $E$ is an open, nonempty set whose closure is contained in $D$,
and that $\varepsilon \in (0, T)$ is fixed. Then, it follows from \cite[Lemma A.1]{Christof2018}
that there exists a function $z \in C_c^\infty( (0, T] \times \Omega)$ that is 
positive in $(\varepsilon, T] \times E$ and zero everywhere in $(0, T] \times \Omega \setminus (\varepsilon, T] \times E$.
If, for such a $z$, we define $\tilde u := (\partial_t z - \Delta z)|_{(0, T) \times D}$, where the vertical bar 
denotes a restriction, then it clearly holds $S(\tilde u) = z(T) > 0$ a.e.\ in $E$.  
From the $C([0, T];L^2(\Omega))$-regularity and the properties of the solutions of  \rev{\eqref{eq:ex4obsprob-def}}
and
the closedness of the set $\{v \in L^2(\Omega) \mid v \geq \psi \text{ a.e.\ in }\Omega\}$ in $L^2(\Omega)$,
we further obtain that $S(\alpha \tilde u) \geq \psi$ has to hold a.e.\ in $\Omega$ for all $\alpha \in \R$. 
In combination with the trivial identity $S(0) = 0$ and $z(T) > 0$ a.e.\ in $E$, it now follows immediately
that
\begin{linenomath}\begin{equation*}
\psi \leq S(\alpha \tilde u)  = S(\alpha \tilde u) - S(0) = \alpha (S(\tilde u) - S(0)) 
=
\alpha S(\tilde u)
=
\alpha z(T)
\end{equation*}\end{linenomath}
cannot be true a.e.\ in $\Omega$ for all $\alpha \in \R$.
This shows that the map $S$ is indeed non-affine 
in the situation of \rev{\eqref{eq:ex4optprob} and \eqref{eq:ex4obsprob-def}}. In summary, we may now again conclude that 
\rev{\eqref{eq:ex4optprob}} satisfies all of the conditions in \cref{ass:standing} (with $p:=2$, $Y := L^2(\Omega)$, $U := L^2(0, T;L^2(D))$, 
and appropriately rescaled norms). \Cref{th:nonunique,th:nonstable} thus apply to \rev{\eqref{eq:ex4optprob}},
and we obtain that this optimal control problem is 
not uniquely solvable for certain choices of the tuple $(y_d, u_d) \in L^2(\Omega) \times L^2(0, T;L^2(D))$
and that the solution set of \rev{\eqref{eq:ex4optprob}} does not admit a continuous selection. 
\end{example}

\rev{Note that, for $\nu \to 0$, the optimal control problem considered in \cref{ex:4}
approaches -- at least formally -- a problem of endpoint controllability for the system \eqref{eq:ex4obsprob-def}, 
cf.\ \cite{Loheac2017,Zuazua2001}.
We would like to emphasize in this context that, in \cref{ex:4}, the 
inequality $y \geq \psi$ is not a state constraint but a part of the evolution 
variational inequality that defines the mapping ${S\colon L^2(0,T;L^2(D)) \to L^2(\Omega)}$.
In particular, all controls $u \in L^2(0,T;L^2(D))$ are admissible in \eqref{eq:ex4optprob}. For a detailed discussion of the 
differences between optimal control problems with pointwise state constraints and 
optimal control problems governed by variational inequalities with unilateral constraint sets,
we refer the reader to \cite[Section 1]{Kunisch2012}.}

\section{\rev{Concluding Remarks}}
\label{sec:4}
\rev{As we have seen in this paper, 
for optimal control problems \ref{eq:problem} in uniformly convex and uniformly smooth spaces
that involve a weak-to-weak continuous control-to-state map $S$, 
the nonlinearity of the considered system dynamics necessarily implies
that there exist examples of desired states and controls 
for which \ref{eq:problem} is nonuniquely solvable and ill-posed in the sense of Hadamard. 
What is important to note in this context is that,
although our results show that such cases exist, they do not make any statement about how often 
they are encountered. To obtain additional information about the size and/or geometric properties of the set of tuples 
$(y_d, u_d)$ for which \ref{eq:problem} possesses more than one solution, 
one can proceed along the lines of \cref{th:nonunique,th:nonstable} and invoke
results on the exceptional sets of metric projections, cf.\ \cite{Westphal1989} and the references therein.
This, however, typically requires additional assumptions. We remark that, similarly, 
it is also possible to generalize the results of \cref{th:nonunique,th:nonstable} to problems 
that involve additional state and/or control constraints  $y \in Y_{\textup{ad}} \subset Y$ and $u \in U_{\textup{ad}} \subset U$
provided these constraints still allow to prove the nonconvexity and weak closedness of the set $M:=
\left \{
(S(u), u) \mid u \in U_{\textup{ad}},\,  S(u) \in Y_{\textup{ad}} 
\right \} \subset Y \times U$. Lastly, we would like to mention that studying the 
(non)uniqueness of solutions of \ref{eq:problem} becomes much more involved 
if $Y$ and $U$ are not assumed to be uniformly smooth and uniformly convex and if the exponent $p$
is also allowed to take the value one. 
(Such cases occur, for instance, in the context of bang-bang  and $L^1$-tracking-type optimal control problems, see \cite{Casas2012:1}
and \cite[Example 3.11]{ChristofVexler2021}.)
On the one hand, in spaces that are not uniformly smooth and uniformly convex,
solutions of problems of the form \ref{eq:problem} can be nonunique even when $S$ is the identity map.
Compare, for instance, with the example 
\begin{linenomath}\begin{equation*}
\min_{y \in \R^2, \, u \in \R^2} \|(1, 0)^T - y\|_\infty^2 + \|u\|_\infty^2 
\quad
\text{s.t.\ } y = u
\end{equation*}\end{linenomath}
involving the $\infty$-norm $\|\cdot\|_\infty$ on $\R^2$ in this context. On the other hand, 
in the absence of uniform convexity and uniform smoothness,
it is also possible that a Chebyshev set is nonconvex and that 
a problem of the form \ref{eq:problem} involving a non-affine $S$ is uniquely solvable for all 
$(y_d, u_d)$, cf.\  \cite[Example 2.11]{Fletcher2015}. 
As a consequence, general purpose results analogous to \cref{th:nonunique,th:nonstable} 
are not available if the assumptions of uniform smoothness and uniform convexity on $Y$ and $U$ are dropped.%
}%

\section*{Acknowledgments} We would like to thank 
Gerd Wachsmuth for making us aware of the concept of Chebyshev sets.


\bibliographystyle{alpha}

\begin{thebibliography}{CMWC18}

\bibitem[AADH20]{AliDeckelnickHinze2020}
A.~Ahmad~Ali, K.~Deckelnick, and M.~Hinze.
\newblock Global minima for optimal control of the obstacle problem.
\newblock {\em ESAIM Control Optim.~Calc.~Var.}, 26:64, 2020.

\bibitem[ABM06]{Attouch2006}
H.~Attouch, G.~Buttazzo, and G.~Michaille.
\newblock {\em Variational Analysis in Sobolev and BV Spaces}.
\newblock SIAM, Philadelphia, 2006.

\bibitem[Bar84]{Barbu1984}
V.~Barbu.
\newblock {\em Optimal Control of Variational Inequalities}.
\newblock Research Notes in Mathematics. Pitman, 1984.

\bibitem[BMRR14]{Betz2014}
T.~Betz, C.~Meyer, A.~Rademacher, and K.~Rosin.
\newblock Adaptive optimal control of elastoplastic contact problems.
\newblock Ergebnisberichte des Instituts f\"{u}r Angewandte Mathematik, TU
  Dortmund, Nr.~496, 2014.

\bibitem[Bro89]{Brown1989}
A.~L. Brown.
\newblock Set valued mappings, continuous selections, and metric projections.
\newblock {\em J.~Approx.~Theory}, 57:48--68, 1989.

\bibitem[BV10]{BorweinVanderwerff2010}
J.~M. Borwein and J.~D. Vanderwerff.
\newblock {\em Convex Functions: Constructions, Characterizations and
  Counterexamples}.
\newblock Cambridge University Press, Cambridge, 2010.

\bibitem[Cas12]{Casas2012:1}
E.~Casas.
\newblock Second order analysis for bang-bang control problems of {PDEs}.
\newblock {\em SIAM J.~Control Optim.}, 50(4):2355--2372, 2012.

\bibitem[Chr19]{Christof2019}
C.~Christof.
\newblock Sensitivity analysis and optimal control of obstacle-type evolution
  variational inequalities.
\newblock {\em SIAM J.~Control Optim.}, 57(1):192--218, 2019.

\bibitem[Cla36]{Clarkson1936}
J.~A. Clarkson.
\newblock Uniformly convex spaces.
\newblock {\em Trans.~Amer.~Math.~Soc.}, 40(3):396--414, 1936.

\bibitem[CMWC18]{Christof2018}
C.~Christof, C.~Meyer, S.~Walther, and C.~Clason.
\newblock Optimal control of a non-smooth semilinear elliptic equation.
\newblock {\em Math.~Control Relat.~Fields}, 8(1):247--276, 2018.

\bibitem[CV21]{ChristofVexler2021}
C.~Christof and B.~Vexler.
\newblock New regularity results and finite element error estimates for a class
  of parabolic optimal control problems with pointwise state constraints.
\newblock {\em ESAIM Control Optim.~Calc.~Var.}, 27(4), 2021.
\newblock published online.

\bibitem[CW20]{ChristofSSC2020}
C.~Christof and G.~Wachsmuth.
\newblock On second-order optimality conditions for optimal control problems
  governed by the obstacle problem.
\newblock {\em Optimization}, 2020.
\newblock to appear.

\bibitem[DZ93]{Dontchev1993}
A.~L. Dontchev and T.~Zolezzi.
\newblock {\em Well-Posed Optimization Problems}.
\newblock Number 1543 in Lecture Notes in Mathematics. Springer, 1993.

\bibitem[FM15]{Fletcher2015}
J.~Fletcher and W.~B. Moors.
\newblock Chebyshev sets.
\newblock {\em J.~Aust.~Math.~Soc.}, 98(2):161--231, 2015.

\bibitem[GLS05]{Gugat2005}
M.~Gugat, G.~Leugering, and G.~Sklyar.
\newblock {Lp}-optimal boundary control for the wave equation.
\newblock {\em SIAM J.~Control Optim.}, 44(1):49--74, 2005.

\bibitem[Haf20]{Hafemeyer2020PhD}
D.~Hafemeyer.
\newblock {\em Optimal Control of the Parabolic Obstacle Problem}.
\newblock PhD thesis, Technische Universit\"at M\"{u}nchen, 2020.

\bibitem[HKST15]{Heinonen2015}
J.~Heinonen, P.~Koselka, N.~Shanmugalingam, and J.~T. Tyson.
\newblock {\em Sobolev Spaces on Metric Measure Spaces}.
\newblock Number~27 in New Mathematical Monographs. Cambridge University Press,
  2015.

\bibitem[HPS07]{Herty2007}
M.~Herty, R.~Pinnau, and M.~Sea\"{i}d.
\newblock Optimal control in radiative transfer.
\newblock {\em Optim.~Methods Softw.}, 22(6):917--936, 2007.

\bibitem[HRUW]{OPTPDE}
R.~Herzog, A.~R{\"o}sch, S.~Ulbrich, and W.~Wollner.
\newblock {OPTPDE} - {A} collection of problems in {PDE}-constrained
  optimization.
\newblock \texttt{http://www.optpde.net}.

\bibitem[HRUW14]{OPTPDE2}
R.~Herzog, A.~R{\"o}sch, S.~Ulbrich, and W.~Wollner.
\newblock {OPTPDE}: A collection of problems in {PDE}-constrained optimization.
\newblock In G.~Leugering, P.~Benner, S.~Engell, A.~Griewank, H.~Harbrecht,
  M.~Hinze, R.~Rannacher, and S.~Ulbrich, editors, {\em Trends in {PDE}
  Constrained Optimization}, volume 165 of {\em International Series of
  Numerical Mathematics}, pages 539--543. Springer, 2014.

\bibitem[KKV00]{Kainen2000}
P.~C. Kainen, V.~K\r{u}rkov\'{a}, and A.~Vogt.
\newblock Geometry and topology of continuous best and near best
  approximations.
\newblock {\em J.~Approx.~Theory}, 105(2):252 -- 262, 2000.

\bibitem[Kle61]{Klee1961}
V.~Klee.
\newblock Convexity of {C}hebyshev sets.
\newblock {\em Math.~Ann.}, 142:292--304, 1961.

\bibitem[KS00]{KinderlehrerStampacchia1980}
D.~Kinderlehrer and G.~Stampacchia.
\newblock {\em An Introduction to Variational Inequalities and Their
  Applications}, volume~31 of {\em Classics in Applied Mathematics}.
\newblock SIAM, 2000.

\bibitem[KW12]{Kunisch2012}
K.~Kunisch and D.~Wachsmuth.
\newblock Sufficient optimality conditions and semi-smooth {N}ewton methods for
  optimal control of stationary variational inequalities.
\newblock {\em ESAIM Control Optim.~Calc.~Var.}, 18(2):520--547, 2012.

\bibitem[LTZ17]{Loheac2017}
J.\ Loh\'{e}ac, E.~Tr\'{e}lat, and E.~Zuazua.
\newblock Minimal controllability time for the heat equation under unilateral
  state or control constraints.
\newblock {\em Math.~Models Methods Appl.~Sci.}, 27(9):1587--1644, 2017.

\bibitem[Meg98]{Megginson1998}
R.~E. Megginson.
\newblock {\em An Introduction to {B}anach Space Theory}.
\newblock Number 183 in Graduate Texts in Mathematics. Springer, 1998.

\bibitem[Mus07]{Muselli2007}
E.~Muselli.
\newblock Affinity and well-posedness for optimal control problems in {H}ilbert
  spaces.
\newblock {\em J.~Convex Anal.}, 14(4):767 -- 784, 2007.

\bibitem[Ne\12]{Necas2012}
J.~Ne\v{c}as.
\newblock {\em Direct Methods in the Theory of Elliptic Equations}.
\newblock Springer, Berlin, 2012.

\bibitem[Pet39]{Pettis1939}
B.~J. Pettis.
\newblock A proof that every uniformly convex space is reflexive.
\newblock {\em Duke Math.~J.}, 5(2):249--253, 1939.

\bibitem[Pig20]{Pighin2020}
D.~Pighin.
\newblock Nonuniqueness of minimizers for semilinear optimal control problems.
\newblock arXiv:2002.04485, 2020.

\bibitem[Sch13]{Schweizer2013}
B.~Schweizer.
\newblock {\em Partielle Differentialgleichungen}.
\newblock Springer, Berlin/Heidelberg, 2013.

\bibitem[WF89]{Westphal1989}
U.~Westphal and J.~Frerking.
\newblock On a property of metric projections onto closed subsets of {H}ilbert
  spaces.
\newblock {\em Proc.~Amer.~Math.~Soc.}, 105(3):644--651, 1989.

\bibitem[Yos80]{Yosida1980}
K.~Yosida.
\newblock {\em Functional Analysis}.
\newblock Springer, 1980.

\bibitem[Zol81]{Zolezzi1981}
T.~Zolezzi.
\newblock A characterization of well-posed optimal control systems.
\newblock {\em SIAM J.~Control Optim.}, 19(5):604--616, 1981.

\bibitem[Zua01]{Zuazua2001}
E.~Zuazua.
\newblock Some results and open problems on the controllability of linear and
  semilinear heat equations.
\newblock In F.~Colombini and C.~Zuily, editors, {\em Carleman Estimates and
  Applications to Uniqueness and Control Theory}, pages 191--211.
  Birkh{\"a}user, 2001.

\end{thebibliography}

\medskip
Received xxxx 20xx; revised xxxx 20xx.
\medskip

\end{document}